\documentclass{article}
\pdfoutput=1
\usepackage{amsmath}
\usepackage{amsthm}
\usepackage{amssymb}
\usepackage{mathrsfs}
\usepackage{graphicx}
\usepackage{verbatim}
\usepackage{tikz}
\usepackage[all]{xy}
\usepackage{graphics}
\usepackage[colorlinks=true]{hyperref}
\hypersetup{urlcolor=blue, citecolor=red}

\newtheorem{theorem}{Theorem}[section]

\newtheorem{proposition}[theorem]{Proposition}
\theoremstyle{definition}

\newcommand{\R}{\mathbb{R}}

\def\beq{\begin{equation}}
\def\eeq{\end{equation}}

\title{\sc A sharp Bernstein-type theorem for entire minimal graphs}
\date{}
\author{Alberto Farina }
\begin{document}
\numberwithin{equation}{section}
\maketitle
{\footnotesize
\centerline{LAMFA, CNRS UMR 7352} 
\centerline {Universit\'e de Picardie Jules Verne}
\centerline{33 rue Saint-Leu, 80039 Amiens, France}}

\centerline{}

\begin{abstract} We consider entire solutions $u$ to the minimal surface equation in $\R^N$, with $ N\ge8,$ and we prove  the following sharp result : {\it if $N-7$ partial derivatives $ \frac{\partial u }{\partial {x_j}}$ are bounded on one side (not necessarily the same), then $u$ is necessarily an affine function. }

\end{abstract}

\smallskip



\section{Introduction and main results}

The graph of a smooth function $ u : \R^N \to \R$ is a minimal surface in $\R^{N+1}$ if and only if $u$ is a solution  to the minimal surface equation 

\beq \label{EqSupMin}
- {\rm div} \left( \frac{\nabla u}{\sqrt{1 + \vert \nabla u \vert^2 }} \right) = 0 \qquad {\rm in} \qquad \R^N, \qquad N\ge2.
\eeq
In his work \cite{Ber} (see also \cite{H}) S.N. Bernstein proved that any smooth solution $u$ to the minimal surface equation in $\R^2$ must be an affine function. 
This result has been extended to $\R^3$ by E. De Giorgi \cite{DG}, to $\R^4$ by  J.F. Almgren \cite{Alm} and, up to dimension $N=7$, by J. Simons \cite{Sim}. On the other hand, in the celebrated paper \cite{BDeGG}, E. Bombieri, E. De Giorgi and E. Giusti proved the existence of a non-affine entire solution of the minimal surface equation \eqref{EqSupMin} for any $ N \ge 8$. Nevertheless, J. Moser \cite{Mos} was able to prove that, if $ \nabla u $ is bounded on $\R^N$, then $u$ must be again an affine function, and this result holds true for every dimension $ N\ge2$. Later, E. Bombieri and E. Giusti \cite{BG} generalized Moser's result by assuming that \emph{only} $N-1$ partial derivatives of $u$ are bounded on $\R^N$, $N\ge2$.
In the recent work \cite{F} the author has further improved this result. More precisely, in \cite{F} it is demonstrated that any smooth solution $u$ to the minimal surface equation \eqref{EqSupMin} with $N-1$ partial derivatives bounded on one side (not necessarily the same) is necessarily an affine function.  


The main result of the present work is the following

\bigskip

\begin{theorem} \label{Th.N-7} Assume $ N \ge 8$ and let $u$ be a solution of the minimal surface equation 
\beq \label{EqSupMin2}
- {\rm div} \left( \frac{\nabla u}{\sqrt{1 + \vert \nabla u \vert^2 }} \right) = 0 \qquad {\rm in} \qquad \R^N.
\eeq

\noindent If $N-7$ partial derivatives of $u$ are bounded on one side (not necessarily the same) then $u$ is an affine function.

\end{theorem}

Theorem \ref{Th.N-7} is sharp. Indeed, it cannot hold true if one assumes that only $N-8$ partial derivatives are bounded on one side. To see this, let us denote by $U=U(x_1,...,x_8)$ the non-affine solution of \eqref{EqSupMin2} in $\R^8$ constructed in \cite{BDeGG} and, for every $ N\ge 8$, set $u_N(x_1,...,x_N) := U(x_1,...,x_8)$. Clearly, $u_N$ solves \eqref{EqSupMin2}, it has $N-8$ partial derivatives that are identically zero, but $u_N$ is not an affine function. 

\medskip

The proof of the previous theorem is based on the following result, which is interesting in its own.

\medskip

\begin{theorem} \label{Splitting}  Let $u$ be a {\underline {non-affine}}  solution of the minimal surface equation 

\beq \label{EqSupMin3}
- {\rm div} \left( \frac{\nabla u}{\sqrt{1 + \vert \nabla u \vert^2 }} \right) = 0 \qquad {\rm in} \qquad \R^N, \qquad N\ge2
\eeq
such that \beq\label{zero-bordo}
u(0) =0.
\eeq
Suppose also that there exists an integer $k \in \{1,...,N\}$ such that, for every $ \alpha =1,...,k,$ the partial derivative $\frac{\partial u}{\partial x_{\alpha}}$ is bounded from below on $\R^N$.


\noindent Then, any blow-down $C \subset \R^{N+1}$ 
associated to the subgraph of $u$ is a minimal cylinder of the form 

\beq\label{conf-cilindro}
C = P \times \R = (\R^{k} \times P') \times \R, 
\eeq
where $P^{'} \subset \R^{N-k}$ is a minimal cone singular at the origin. 
\end{theorem}

\section{Proofs}

Let us first consider Theorem 1.2. 

{\begin{proof} [Proof of Theorem \ref{Splitting}] 
Let $U$ be the subgraph of $u$ and $U_j$ the one of the functions $u_j (x) = \frac{u(jx)}{j}$, where $x \in \R^N$ and $ j \ge 1$. 

By assumption, $U$ is a non-trivial set of least perimeter in $\R^{N+1}$ with $0 \in \partial U$.
Therefore, a classical procedure (see e.g. Theorem 17.3 and Theorem 9.3 of \cite{G}) provides a minimal cone $C \subset \R^{N+1}$ (with vertex at the origin), as the limit of a {\it subsequence} of $U_j$ (still denoted by $U_j$) with respect to the $L^1_{loc}$ convergence. The minimal cone $C$ is usually called a blow-down of $U$ and we have that, for almost every $ R>0$,

\beq\label{stimaC}
\omega_N \le R^{1 -(N+1)} \int_{B(0,R)} \vert D {\bf 1}_{C} \vert.
\eeq
where $\omega_N$ denotes de volume of the unit ball of $ \R^N$. Indeed, by the minimality of the sets $U_j$, the monotonicity of the functions $ R \to R^{1 -(N+1)} \int_{B(0,R)} \vert D {\bf 1}_{U_j} \vert $ and the fact that $ 0 \in \partial U_j$, the following well-known density estimates hold (see for instance formula (5.16) on p. 72 of \cite{G} or formula (1.13) on p.2 of \cite{GMM})  

\beq\label{stima1}
\omega_N \le (Rj)^{1 -(N+1)} \int_{B(0,Rj)} \vert D {\bf 1}_U \vert =R^{1 -(N+1)} \int_{B(0,R)} \vert D {\bf 1}_{U_j} \vert. 
\eeq
Therefore 

\beq\label{stima2}
\omega_N \le R^{1 -(N+1)} \int_{B(0,R)} \vert D {\bf 1}_{U_j} \vert  \to  R^{1 -(N+1)} \int_{B(0,R)} \vert D {\bf 1}_{C} \vert,
\eeq
proving \eqref{stimaC}. From \eqref{stimaC} we get that $0 \in \partial C$ and, by Lemma 16.3 and Proposition 16.5  of \cite{G}, we also know that $C$ is itself a subgraph of a generalized solution to the minimal surface equation $v : \R^N \to [-\infty, +\infty]$ (also called quasi-solution to the minimal surface equation) (see 
\cite{Mir}, \cite{MasMir}, \cite{G}). 

Also the set
\beq\label{P}
P = \{ x \in \R^N\, : \, v(x) = + \infty \, \}
\eeq


\noindent must be non-empty since $u$ is non-affine. To see this we follow \cite{G}. If $P = \emptyset$ then, by Lemma 17.7 of \cite {G}, the family of functions $u_j$ is equibounded from above on compact sets of $ \R^N$. The latter and the definition of $u_j$ immediately provides that 
\beq\label{equi-u}
\sup_{B(0,j)} u \le {\cal K} j   
\eeq
for some constant ${\cal K}>0$. 
On the other hand, the celebrated gradient estimate of \cite{BDeGMir} tells us that 
\beq\label{stimaGrad}
\forall \, x \in \R^N, \,\, \forall R>0,  \qquad \vert \nabla u(x) \vert \le C_1 exp \Big [ C_2 \Big ( 
\frac{{\sup_{B(x,R)} u - u(x)}}{R} \Big )\Big ]   
\eeq
 where $C_1$ and $C_2$ are positive constant depending only on the dimension $N$.  
Now, combining \eqref{equi-u},\eqref{stimaGrad} and letting $j \to +\infty$ we obtain that $\vert \nabla u \vert \in L^{\infty}(\R^N)$.  Thus, we can apply the result of Moser \cite{Mos}, already recalled in the introduction, to get that $u$ is an affine function. The latter is impossible since we are supposing that $u$ is not affine. 

We also remark that 
\beq\label{PperRdentroC}
P \times \R \subset C
\eeq
by construction, and so $ P \not\equiv \R^N$, since $C \not\equiv \R^{N+1}$.

Also, $P$ is a minimal cone in $\R^N$,  with vertex at the origin (since $C$ is a minimal cone with vertex at the origin). Combining the two latter pieces of information we also get that the origin of $\R^N$ belongs to $ \partial P$ and so $0 \in \partial (P\times \R)$.

Next we observe that $P$ is singular at the origin of $\R^N$. Suppose not, then $P$ would be a half-space of $\R^N$ and thus the minimal cone $C$ would contain the half-space $P\times \R$ which, in turn, would give $C  \equiv P\times \R$ (see for instance Theorem 15.5 of \cite{G} or \cite{MirPMax}) and so 

\beq\label{C-semispazio}
\forall \, R >0 \qquad \omega_N = R^{1 -(N+1)} \int_{B(0,R)} \vert D {\bf 1}_{C} \vert
\eeq

Therefore, by \eqref {C-semispazio}, the monotonicity of $ R \to R^{1 -(N+1)} \int_{B(0,R)} \vert D {\bf 1}_U \vert $ and \eqref{stima1}-\eqref{stima2} we obtain

\beq\label{U-semispazio}
\forall \, R >0 \qquad \omega_N = R^{1 -(N+1)} \int_{B(0,R)} \vert D {\bf 1}_{U} \vert
\eeq
which proves that $U$ is a half-space and $ U = C$ (by the construction of $C$). 

But $ U = C  \equiv P\times \R$ implies that $ \partial U$ is a vertical hyperplane, contradicting  the fact $\partial U$ is the graph of the function $u$. 


Moreover, by the discussions above, we have that the minimal cone $ P \times \R$ satisfies $P \times \R \subset C$ and $0 \in \partial (P\times \R) \cap \partial C$. Therefore, $ P \times \R$ 
must coincide with the minimal cone $C$ (cf. \cite{Moschen} or Theorem 2.4 of \cite{GMM}).  This provides the first equality in \eqref{conf-cilindro} with $P$ a minimal cone singular at the origin.

\medskip

To conclude we observe that the assumption on the partial derivatives implies the existence of a constant $ K>0$ such that for every $j \ge 1$, every $x \in \R^N$ and every $\alpha = 1,...,k$
\beq\label{j-monot}
u_j(x + e_{\alpha}) - u_j(x) = \frac{1}{j} \int_0^j \frac{\partial u}{\partial x_{\alpha}} (jx_1,..., jx_{\alpha}+t,...,jx_N) dt  \ge -K, 
\eeq
where $e_{1},...,e_{k}$ denote the first $k$ vectors of the standard (or natural) basis for $ \R^N$.   
Now recall that, up to a subsequence, $(u_j)$ converges almost everywhere to $v$ on $\R^N$ (cf. for instance Proposition 16.5  of \cite{G}) so, by letting $j\to +\infty $ in \eqref{j-monot}, we get that for almost every $x \in P$ the point $x +e_{\alpha} $ also belongs to $P$, for every $ \alpha=1,...,k$. 

Therefore we have  
\beq\label{Harnack-Bombieri-Giusti}
\forall \, \alpha=1,...,k \qquad P + e_{\alpha}  \subseteq P.  
\eeq
By applying Proposition \ref{LemmaGMM} we deduce that $P$ is a cylinder in the directions $e_1,..., e_k$, for otherwise, $P$ would be a half-space of $\R^N$ contradicting the fact that $P$ is singular at the origin.


Since $P$ is a cylinder in the directions $e_1,..., e_k$ we must have $P =\R^{k} \times P',$ with $P^{'} \subset \R^{N-k}$ minimal (since $P$ is minimal in $\R^N$) and $P^{'}$ singular at the origin of $\R^{N-k}$ (since $P$ is singular at the origin of $\R^N$). This concludes the proof. 



\end{proof}

Now we can prove the main result.

{\begin{proof} [Proof of Theorem \ref{Th.N-7}] The function $v(x) = u(x) - u(0) $ is a solution of \eqref{EqSupMin3} with $v(0) =0$ and with $N-7$ partial derivatives bounded on one side (not necessarily the same).

Up to change the variable $ x_j $ in $- x_j$ (if necessary), we can suppose that $v$ has $N-7$ partial derivatives which are bounded from below.  Also, we can assume that those partial derivatives are taken with respect to the first $N-7$ variables.  If $v$ were not linear then, by applying Theorem \ref{Splitting} with $k=N-7$, we would get 

\beq\label{conf-cilindro-N-7}
C = P \times \R = (\R^{N-7} \times P^{'}) \times \R, 
\eeq

\bigskip

\noindent where $P^{'} \subset \R^7$ would be a minimal cone singular at the origin.  But this is is impossible since all the (non-trivial) minimal cones in dimension $ n \le 7$ are half-spaces. This gives that $v$ is a linear function and so $u$ must be affine, as desired. 

\end{proof}

\bigskip
Finally we state and proof Proposition \ref{LemmaGMM} used in the last part of the proof of Theorem \ref{Splitting}. This result is essentially contained in (the proof of) Lemma 2.3 of \cite{GMM}. 

\bigskip

\begin{proposition} \label{LemmaGMM} Assume $k \ge 2$ and let ${\cal C}$ be a non-trivial minimal cone in $\R^k$ such that $ {\cal C} + v \subseteq {\cal C},$ for some $ v \in \R^k 
\setminus \{ 0 \}$. Then, either $ {\cal C} + v \neq {\cal C}$ and ${\cal C} $ is a half-space or,
$ {\cal C} + v = {\cal C}$ and ${\cal C}$ is a cylinder in the direction $v$. 
\end{proposition}

\begin{proof}
Same proof of Lemma 2.3 of \cite{GMM}. 
\end{proof}





\vspace{0.5cm}

\noindent \textbf{Acknowledgements: } The author thanks E. Valdinoci and L. Mazet for a careful reading of a first version of this article. The author is partially supported by the ERC grant EPSILON ({\it Elliptic Pde's and Symmetry of Interfaces and Layers for Odd Nonlinearities}) and by the ERC grant COMPAT ({\it Complex Patterns for Strongly Interacting Dynamical Systems}).

\end{document}